\documentclass{amsart}
\usepackage{amssymb,amsmath, amscd, mathrsfs, graphics,xcolor,graphicx}

\newtheorem{claim}{}[section]
\newtheorem{defn}[claim]{Definition}

\newtheorem{thm}[claim]{Theorem}

\newtheorem{remark}[claim]{Remark}
\newtheorem{prop}[claim]{Proposition}

\newtheorem{cor}[claim]{Corollary}

\title[Geometric properties of the subordination function]{Some geometric properties of the subordination 
function associated to an operator-valued free convolution semigroup}

\author{Serban Teodor Belinschi}
\address{CNRS - Institut de Math\'ematiques de Toulouse\\
118 Route de Narbonne\\ 31062 Toulouse, FRANCE}
\email{serban.belinschi@math.univ-toulouse.fr}
\thanks{}

\begin{document}

\begin{abstract}
In his article {\em On the free convolution with a semicircular 
distribution}, Biane found very useful characterizations of the 
boundary values of the imaginary part of the Cauchy-Stieltjes transform 
of the free additive convolution of a probability measure on $\mathbb R
$ with a Wigner (semicircular) distribution. Biane's methods were 
recently extended by Huang to measures which belong to the partial 
free convolution semigroups introduced by Nica and Speicher. This
note further extends some of Biane's methods and results to free 
convolution powers of operator-valued distributions and to free 
convolutions with operator-valued semicirculars. In addition, it
investigates properties of the Julia-Carath\'eodory derivative of
the subordination functions associated to such semigroups, 
extending results from \cite{BB}.
\end{abstract}

\maketitle

\section{Introduction}

Free probability, introduced by Voiculescu in order to 
study free group factors, gained considerable importance
after the discovery in \cite{VInv} of the connection between
freeness and the asymptotic behaviour of large random matrices.
One of the most significant consequences of the main result of
\cite{VInv} is the fact that two independent selfadjoint random 
matrices $H_N,A_N$ - $H_N$ being a gaussian matrix - are asymptotically
free as $N\to\infty$. Thanks to Wigner's work, it is known since the
'50s that the asymptotic distribution of $H_N$ as $N\to\infty$ is
the semicircle law. In particular, the distribution of 
$A_N+H_N$ is modeled by Voiculescu's free additive 
convolution \cite{DVJFA} of a standard semicircular 
distribution with the limiting distribution of $A_N$. 
In \cite{Biane2}, this convolution is analyzed in great
detail: among others, a formula for the density of the
corresponding distribution is provided, and it is shown 
that this density is bounded, continuous and analytic wherever
positive. However, in order to study the asymptotic eigenvalue distribution of more general selfadjoint 
polynomials $P(A_N,H_N)$ it is necessary to consider the more general framework of free convolutions
of {\em operator-valued distributions} \cite{V*,S1,S2,HRS,BMS}.
In the present note, we find certain operator-valued
counterparts of Biane's results from \cite{Biane2};
necessarily, several of the conclusions of \cite{Biane2}
do not hold in this more general setting. 

As it is shown in \cite{ABFN}, there exists an intimate connection
between the free additive convolution with an operator-valued 
semicircular distribution and the free convolution powers of
operator-valued distributions. In particular, it turns out 
that the analytic tools used in the study of free convolution
powers of operator-valued distributions are generalizations of
the analytic tools used in the study of the free convolution with an
operator-valued semicircular distribution. Thus, we 
write our proofs in the more general context. This
has the advantage of allowing us to draw conclusions 
about more general free additive convolutions of operator-valued
distributions.

The second section is dedicated to introducing the main
concepts and tools we require. We state and prove our main
results in the third and fourth section. 

\section{Noncommutative functions, distributions and convolutions}

\subsection{Noncommutative probability spaces and distributions}
Following D. Voiculescu \cite{DVJFA,V*}, by a noncommutative 
probability space we understand a pair $(\mathcal A, \tau)$ where 
$\mathcal A$ is a unital $\ast$-algebra over $\mathbb C$ 
and $\tau\colon\mathcal A\to\mathbb{C}$ is a positive linear functional 
with $\tau(1)=1$. Let $\mathcal B$ be a unital C$^\ast$-algebra. 
A $\mathcal B$-valued non-commutative probability space is a triple 
$(\mathcal A,\mathbb E_\mathcal B,\mathcal B)$, where $\mathcal A$ 
is a unital $\ast$-algebra containing $\mathcal B$ as a 
$\ast$-subalgebra and $\mathbb E_\mathcal B$ is a unit-preserving 
positive conditional expectation from $\mathcal A$ onto $\mathcal B$ 
(in particular, the units of $\mathcal A$ and $\mathcal B$ coincide). 
If $\mathcal B\subset\mathcal A$ is an inclusion of unital 
C${}^*$-algebras, then we call $(\mathcal A,\mathbb E_\mathcal B,
\mathcal B)$ a $\mathcal B$-valued noncommutative 
C${}^*$-probability space. For simplicity, we will suppress the 
subscript of $\mathbb E_\mathcal B$ whenever there is no risk of 
confusion, and denote our conditional expectation by $\mathbb E$. 
Elements $X\in\mathcal A$ are called {\emph{random variables} or 
(in the second context) {\emph{$\mathcal B$-valued} (or 
\emph{operator-valued}) \emph{random variables}.

We use the notation
$\mathcal B\langle\mathcal X_1,\mathcal X_2,\dots\mathcal X_n\rangle$ 
for the $\ast$-algebra freely generated by $\mathcal B$ and the 
noncommuting selfadjoint symbols $\mathcal X_1,\mathcal X_2,\dots,
\mathcal X_n$. If $X\in\mathcal A$ is a selfadjoint element, then we will also use the notation $\mathcal B\langle X\rangle$ for the 
$\ast$-algebra generated by $X$ and $\mathcal B$. Following \cite{ABFN}
we denote the set of all positive, unit preserving, conditional 
expectations from $\mathcal B\langle\mathcal X\rangle$ to $\mathcal B$ 
by $\Sigma(\mathcal B)$. 
Given $\mu\in\Sigma(\mathcal B)$, its $n^{\rm th}$ moment is the
$n-1$-linear map $\mu_n\colon\mathcal B\times\cdots\times\mathcal B
\to\mathcal B$ defined by $\mu_n(b_1,b_2,\dots,b_{n-1}):=
\mu\left[\mathcal Xb_1\mathcal Xb_2\cdots\mathcal Xb_{n-1}\mathcal 
X\right].$ We define the zeroth moment to be $1\in\mathcal B$
and the first moment to be $\mu\left[\mathcal X\right]\in\mathcal B$.
We also denote  $\Sigma_0(\mathcal B)$ the set of all 
$\mu\in\Sigma(\mathcal B)$ whose moments do not grow faster than 
exponentially, that is, all $\mu\in\Sigma(\mathcal B)$ for which
there exists some $M>0$ such that, for all  
positive integers $m$, all $b_1,\dots,b_n\in M_m(\mathcal B)$ and 
$\mathcal X_m=\mathcal X\otimes1_m$ we have that
$$
\|(\mu\otimes{\rm Id}_m)(\mathcal X_mb_1\mathcal X_mb_2\cdots\mathcal
X_m b_n\mathcal X_m)\|< M^{n+1}\|b_1\|\cdots\|b_n\|.
$$ 
If $(\mathcal A,\mathbb E,\mathcal B)$ is a $\mathcal B$-valued 
noncommutative probability space and $X=X^*\in\mathcal A$, we define
its distribution with respect to $\mathbb E$ to be the element 
$\mu_X\in\Sigma_0(\mathcal B)$ satisfying
$$
\mu_X(P(\mathcal X))=\mathbb E(P(X))\ \text{for all}\ P(\mathcal X)\in
\mathcal B\langle\mathcal X\rangle.
$$
If $X$ belongs to a $\mathcal B$-valued 
noncommutative C${}^*$-probability space, then $\mu_X\in\Sigma_0(
\mathcal B)$. Conversely, as shown by Voiculescu in \cite{V*}, if
$\mu\in\Sigma_0(\mathcal B)$, then there exist a $\mathcal B$-valued 
C${}^*$-noncommutative probability space containing an element $X=X^*$ 
such that $\mu_X=\mu$. In the simpler case $\mathcal B=\mathbb C$,
$\mu_X$ can be identified with a Borel probability measure
supported on the compact set $\sigma(X),$ the spectrum of $X$ (see 
\cite{akhieser}).

\subsection{Free independence and relevant transforms}
We present next the free independence, and define the relevant analytic 
transforms, in a C${}^*$-algebraic context, as this is the context that 
is considered most often in this paper. 

\begin{defn}\label{freeness}
Let $(\mathcal A, \mathbb E,\mathcal B)$ be a $\mathcal B$-valued 
noncommutative C${}^*$-probability space and $\{X_i\}_{i\in I}$ be a 
family of selfadjoint elements from $\mathcal A$.
The family $\{X_i\}_{i\in I}$ is said to be \emph{freely independent}
(or just {\em free}) over $\mathcal B$ with respect to $\mathbb E$ 
if for any $n\in\mathbb N$, $\mathbb E[A_1\cdots A_n]=0$ whenever 
$A_j\in\mathcal B\langle X_{\iota(j)}\rangle\cap\ker(\mathbb E)$, 
$\iota(j)\in I$, $\iota(k)\neq\iota(k+1)$ for all 
$k\in\{1,\dots,n-1\}$. 
\end{defn}
If $X, Y$ are two freely independent $\mathcal B$-valued 
noncommutative random variables, then $\mu_{X+Y}$ depends only on 
$\mu_X$ and $\mu_Y$ and is called the free additive convolution 
of $\mu_X$ and $\mu_Y$. We denote $\mu_{X+Y}$ by $\mu_X\boxplus
\mu_Y$.

It is natural to denote $\underbrace{\mu\boxplus\cdots\boxplus\mu}_{n
\text{ times}}$ by $\mu^{\boxplus n}$. Obviously, $\{\mu^{\boxplus n}|n
\in\mathbb N\}$ forms a discrete semigroup. A remarkable result of Nica 
and Speicher \cite{NS} states that for any Borel probability measure 
$\mu$ on $\mathbb R$, there exists a partial semigroup, i.e. a family 
$\{\mu^{\boxplus t}\colon t\ge1\}$ such that $\mu^{\boxplus 1}=\mu$ 
and $\mu^{\boxplus s+t}=\mu^{\boxplus s}\boxplus\mu^{\boxplus t}$, 
$s,t\ge1$ (see also \cite{BVSuper}). This result has been extended 
by Curran \cite{Curran} to certain operator-valued distributions.
However, as it will be seen below, in the operator-valued context,
analytic transforms indicate that it should be possible - or rather
natural - to consider convolution powers indexed by a subset of the
set of completely positive self-maps of $\mathcal B$. The main result 
of \cite{ABFN} states precisely that: given $\mu\in\Sigma_0(\mathcal B)
$, there exists a family
$$
\left\{\mu^{\boxplus\alpha}|\alpha\colon\mathcal B\to\mathcal B
\text{ completely positive, }\alpha-\text{Id}_\mathcal B
\text{ completely positive}\right\}\subset\Sigma_0(\mathcal B)
$$
such that $\mu^{\boxplus\text{Id}_\mathcal B}=\mu$ and
$\mu^{\boxplus\alpha+\beta}=\mu^{\boxplus\alpha}\boxplus
\mu^{\boxplus\beta}$. Moreover, as shown in 
\cite[Corollary 7.6]{ABFN}, whenever $\mu\in\Sigma_0(\mathcal B)$
is $\boxplus$-infinitely divisible, $\mu^{\boxplus\alpha}\in
\Sigma_0(\mathcal B)$ for {\em any} completely positive $\alpha
\colon\mathcal B\to\mathcal B$.

For the computation of free convolutions, Voiculescu 
\cite{DVJFA,V*} introduced the $R$-transform. In order
to define it, let
\begin{equation}\label{G}
G_\mu(b)=\mu\left[(b-\mathcal X)^{-1}\right], \quad
\Im b>0.
\end{equation}
Here we denote $\Im b=(b-b^*)/2i$, $\Re b=(b+b^*)/2$,
and we write $a>0$ if $a=a^*$ and $\sigma(a)\subset(0,+\infty)$.
The notation $a\ge0$ is used when we require only that
$a=a^*$ and $\sigma(a)\subset[0,+\infty)$.
If $\mu\in\Sigma_0(\mathcal B)$, then $M_\mu(b)=
\mu\left[(1-b\mathcal X)^{-1}b\right]=G_\mu(b^{-1})$
has an analytic continuation to a neighbourhood of zero
and maps $0$ to itself. A simple computation shows that
$M_\mu'(0)=\text{Id}_\mathcal B$, so that, by the inverse
function theorem for Banach spaces, $M_\mu$ has a unique 
compositional inverse, denoted by $M_\mu^{\langle-1\rangle}$, 
around zero which maps zero to itself. Thus,
both $b^{-1}M_\mu^{\langle-1\rangle}(b)$ and 
$M_\mu^{\langle-1\rangle}(b)b^{-1}$ are analytic around zero
and map zero to one.
The $R$-transform of $\mu$ is defined via the formula
$bR_\mu(b)=(M_\mu^{\langle-1\rangle}(b)b^{-1})^{-1}-1$. 
We prefer a slightly different form of the definition 
of $R_\mu$, namely
\begin{equation}\label{R}
R_\mu(b)=G^{\langle-1\rangle}_\mu(b)-b^{-1}.
\end{equation}
This formula is well-defined on an open set which has zero in
its closure, and thus determines $R_\mu$.
The essential property of the $R$-transform, found by
Voiculescu, is that
$$
R_\mu(b)+R_\nu(b)=R_{\mu\boxplus\nu}(b)
$$
on a sufficiently small neighbourhood of zero in $\mathcal B$. Clearly 
then, for any linear completely positive map $\alpha\colon\mathcal B
\to\mathcal B$ such that $\alpha-\text{Id}_\mathcal B$ is still 
completely positive, $\mu^{\boxplus\alpha}$ will be given by
\begin{equation}\label{alpha}
R_{\mu^{\boxplus\alpha}}(b)=\alpha(R_\mu(b)),
\end{equation}
on a neighbourhood of zero. It has been shown in
\cite{ABFN} that such a $\mu^{\boxplus\alpha}\in\Sigma_0
(\mathcal B)$ exists for any $\mu\in\Sigma_0
(\mathcal B)$. A different, simpler proof of this 
result is given in \cite{S3}, where it is also shown that the 
requirement that $\alpha-\text{Id}_\mathcal B$ is itself completely
positive cannot be generally omitted.

It is quite obvious from \eqref{G} that $G_\mu$ plays a role
similar to that of the Cauchy-Stieltjes transform in classical
probability. However, unlike the classical Cauchy-Stieltjes transform,
$G_\mu$ alone does not generally encode all of the distribution $\mu$, 
but only its symmetric moments. It has been a crucial insight of 
Voiculescu that $G_\mu$ is just the first level of a noncommutative
function that {\em does} encode all of $\mu$: this will be outlined
in the next subsection.

\subsection{Noncommutative functions and transforms}
In this subsection we largely follow \cite{BPV1,popavinnikov}
in describing the noncommutative extensions of the analytic transforms
introduced in the previous subsection, and \cite{ncfound} in the
definition of noncommutative sets and functions. We refer to these
three articles and \cite{V2,V1} for details on, and proofs
of, the statements below.

If $S$ is a nonempty set, we denote by $M_{m\times n}
(S)$ the set of all matrices with $m$ rows and $n$ columns having 
entries from $S$. For simplicity, we let $M_n(S):=M_{n\times n}(S)$. 
Given C${}^*$-algebra $\mathcal B$, a {\em noncommutative set} is a 
family $\Omega:=(\Omega_n)_{n\in\mathbb N}$
such that
\begin{enumerate}
\item[(a)] for each $n\in\mathbb N$, $\Omega_n\subseteq M_n
(\mathcal B);$
\item[(b)] for each $m,n\in\mathbb N$, we have $\Omega_m\oplus
\Omega_n\subseteq\Omega_{m+n}$.
\end{enumerate}
The noncommutative set $\Omega$ is called {\em right admissible} if in
addition the condition (c) below is satisfied:
\begin{enumerate}
\item[(c)] for each $m,n\in\mathbb N$ and $a\in\Omega_m,b\in
\Omega_n,w\in M_{m\times n}(\mathcal B)$, there is an $\epsilon>0$
such that $\left(\begin{array}{cc}
a & zw\\
0 & b\end{array}\right)\in\Omega_{m+n}$ for all $z\in\mathbb C,
|z|<\epsilon$.
\end{enumerate}
Given $C^*$-algebras $\mathcal{B,C}$ and a noncommutative
set $\Omega\subseteq\coprod_{n=1}^\infty M_n(\mathcal B)$, a {\em 
noncommutative function} is a family $f:=(f_n)_{n\in\mathbb N}$
such that $f_n\colon\Omega_n\to M_n(\mathcal C)$ and
\begin{enumerate}
\item $f_m(a)\oplus f_n(b)=f_{m+n}(a\oplus b)$ for all
$m,n\in\mathbb N$, $a\in\Omega_m,b\in\Omega_n$;
\item for all $n\in\mathbb N$, $f_n(T^{-1}aT)=T^{-1}f_n(a)T$ whenever
$a\in\Omega_n$ and $T\in GL_n(\mathbb C)$ are such that $T^{-1}aT$
belongs to the domain of definition of $f_n$.
\end{enumerate}
A remarkable result (see \cite[Theorem 7.2]{ncfound}) states that, 
under very mild conditions on $\Omega$, local boundedness for $f$ 
implies each $f_n$ is analytic as a map between Banach spaces.

As mentioned in the previous section, the function $G_\mu$
encodes only the symmetric part of the distribution $\mu$.
It was an extremely important remark of Voiculescu that
$G_\mu$ has a noncommutative extension:
\begin{equation}\label{Gfull}
G_\mu^{[n]}(b)=(\mu\otimes\text{Id}_n)\left[(b-\mathcal X\otimes1_n
)^{-1}\right],\quad n\in\mathbb N.
\end{equation}
There are two noncommutative sets which are natural domains of 
definition for $(G_\mu^{[n]}(b))_{n\in\mathbb N}$ and for $(G_\mu^{[n]}
(b^{-1}))_{n\in\mathbb N}$, respectively: the noncommutative operator 
upper half-plane $(\mathbb H^+(M_n(\mathcal B)))_{n\in\mathbb N}$, 
where $\mathbb H^+(M_n(\mathcal B))=\{b\in M_n(\mathcal B)\colon\Im b>0
\}$, and the set of nilpotent matrices with entries from $\mathcal B$, 
respectively. 
Remarkably, as shown in \cite{V2}, $G_\mu^{[n]}$ maps $\mathbb H^+(M_n
(\mathcal B))$ into $\mathbb H^-(M_n(\mathcal B)):=-\mathbb H^+(M_n
(\mathcal B))$ and $G_\mu^{[n]}(b^*)=G_\mu^{[n]}(b)^*$. It is clear 
that the restriction of $(G_\mu^{[n]})_{n\in\mathbb N}$ to either of 
these two noncommutative sets determines 
$(G_\mu^{[n]})_{n\in\mathbb N}$. For a description of 
how to explicitly recover $\mu$ from $(G_\mu^{[n]})_{n\in\mathbb N}$
via its moments, we refer to \cite{BPV1,popavinnikov}.

It follows from its definition that the $R$-transform has
itself a noncommutative extension, which determines $\mu$ uniquely.
The level-one relation 
\eqref{alpha} extends to $R_{\mu^{\boxplus\alpha}}^{[n]}(b)=(\alpha
\otimes\text{Id}_n)(R_{\mu}^{[n]}(b))$ for $b\in M_n(\mathcal B)$ of 
small enough norm. From this formula and the noncommutative extension
of \eqref{R} we obtain, by adding $b^{-1}$, the relation
$\left(G_{\mu^{\boxplus\alpha}}^{[n]}\right)^{\langle-1\rangle}(b)=
(\alpha\otimes\text{Id}_n)\left(\left(G_{\mu}^{[n]}\right)^{\langle-1
\rangle}(b)\right)-(\alpha\otimes\text{Id}_n-\text{Id}_\mathcal B
\otimes\text{Id}_n)(b^{-1}).$ Replacing $b$ by 
$G_{\mu^{\boxplus\alpha}}^{[n]}(b)$
provides $b=(\alpha\otimes\text{Id}_n)
\left(G_{\mu}^{[n]}\right)^{\langle-1\rangle}
\left(G_{\mu^{\boxplus\alpha}}^{[n]}(b)\right)-
(\alpha\otimes\text{Id}_n-\text{Id}_\mathcal B
\otimes\text{Id}_n)\left(G_{\mu^{\boxplus\alpha}}^{[n]}(b)^{-1}\right)
.$ With the notations
\begin{eqnarray}
F_{\mu}^{[n]}(b)&=&G_{\mu}^{[n]}(b)^{-1},\quad b\in\mathbb H^+
(M_n(\mathcal B)),\ n\in\mathbb N,\label{F}\\
h_{\mu}^{[n]}(b)&=&F_{\mu}^{[n]}(b)-b,\quad b\in\mathbb H^+
(M_n(\mathcal B)),\ n\in\mathbb N,\label{h}
\end{eqnarray}
and
\begin{equation}\label{7}
\omega_\alpha^{[n]}(b)=\left(G_{\mu}^{[n]}\right)^{\langle-1\rangle}
\left(G_{\mu^{\boxplus\alpha}}^{[n]}(b)\right),
\end{equation}
we re-write equation \eqref{alpha} as
\begin{equation}\label{omega}
\omega_\alpha^{[n]}(b)=b+\left[(\alpha-\text{Id}_\mathcal B)\otimes
\text{Id}_n\right]h_{\mu}^{[n]}\left(\omega_\alpha^{[n]}(b)\right),
\quad b\in\mathbb H^+(M_n(\mathcal B)),\ n\in\mathbb N,
\end{equation}
with $\omega_\alpha^{[n]}\colon\mathbb H^+(M_n(\mathcal B))\to
\mathbb H^+(M_n(\mathcal B))$.
The above argument for the existence of $(\omega_\alpha^{[n]})_{n\in
\mathbb N}$ is obviously not complete: for the rigorous
proof, we refer to \cite[Theorem 8.4]{ABFN}. This same theorem
also states that for any $b\in\mathbb H^+(M_n(\mathcal 
B))$, $\omega_\alpha^{[n]}(b)\in\mathbb H^+(M_n(\mathcal B))$
is the unique attracting fixed point of the map
$f_b^{[n]}\colon \mathbb H^+(M_n(\mathcal B))\to
\mathbb H^+(M_n(\mathcal B))$, $f_b^{[n]}(w)=b+
\left[(\alpha-\text{Id}_\mathcal B)\otimes
\text{Id}_n\right]h_{\mu}^{[n]}(w)$, and the right inverse
of the map $H^{[n]}\colon\mathbb H^+(M_n(\mathcal B))\to M_n(
\mathcal B)$, $H^{[n]}(w)=w-\left[(\alpha-\text{Id}_\mathcal B)\otimes
\text{Id}_n\right]h_{\mu}^{[n]}(w)$. 

Of importance in our analysis will be the following result of Popa
and Vinnikov \cite[Theorem 6.6]{popavinnikov}, re-phrased in terms
of the noncommutative function $h$:
\begin{thm}\label{pv}
Let $\mu\in\Sigma_0(\mathcal B)$ be given. Then there exists
a linear map $\eta_\mu\colon\mathcal B\langle\mathcal X\rangle\to
\mathcal B$ and $M\in(0,+\infty)$ such that for any $k\in\mathbb N$, 
$x_1,\dots,x_k\in\mathcal B\langle\mathcal X\rangle$, we have
$$
\left(\eta_\mu\left[x_j^*x_i\right]\right)_{i,j=1}^k\ge0\quad\text{in }
M_k(\mathcal B),
$$
$$
\|\eta_\mu\left[\mathcal Xb_1\mathcal Xb_2\cdots\mathcal Xb_k\mathcal X
\right]\|<M^{n+1}\|b_1\|\|b_2\|\cdots\|b_n\|\quad\text{for all }
b_1,\dots,b_n\in\mathcal B,
$$
and 
$$
h_\mu^{[n]}(b)=(\eta_\mu\otimes{\rm Id}_n)\left[(
\mathcal X\otimes 1_n-b)^{-1}\right]-(\mu\otimes{\rm Id}_n)
(\mathcal X\otimes 1_n),\quad b\in\mathbb H^+(M_n(\mathcal B)),n\in
\mathbb N.
$$
\end{thm}
In \cite{popavinnikov} it is shown that under the assumption that
$\mu\in\Sigma_0(\mathcal B)$, $\mathcal B\langle\mathcal X\rangle$
has a natural C${}^*$-algebra completion, and then the
first statement of the theorem about the norm-bounded $\eta_\mu$ 
becomes equivalent to its complete positivity. This will be
important in our proofs. We finally write \eqref{omega} as
\begin{eqnarray}\label{eta}
\omega_\alpha^{[n]}(b)&=&b-\left[(\alpha-\text{Id}_\mathcal B)\otimes
\text{Id}_n\right](\mu\otimes\text{Id}_n)(\mathcal X\otimes1_n)\\
& & \mbox{}+\left[(\alpha-\text{Id}_\mathcal B)\otimes
\text{Id}_n\right](\eta_\mu\otimes\text{Id}_n)\left[
\left(\mathcal X\otimes1_n-\omega_\alpha^{[n]}(b)\right)^{-1}\right],
\nonumber
\end{eqnarray}
for $b\in\mathbb H^+(M_n(\mathcal B)),\ n\in\mathbb N.$
This equation determines $(\omega_\alpha^{[n]})_{n\in\mathbb N}$ and
thus, via the relation $G_{\mu}^{[n]}\circ\omega_\alpha^{[n]}=
G_{\mu^{\boxplus\alpha}}^{[n]}$,equivalent to \eqref{7}, determines 
$\mu^{\boxplus\alpha}$ in terms of $\mu$ and $\alpha$.

We conclude this section with a simple remark in light of 
\cite{ABFN,V*,RS2}: assume that in equation \eqref{eta} above, $\mu
(\mathcal X)=0$ and $\nu:=\eta_\mu$ is a conditional expectation. 
Denote $\beta:=\alpha-\text{Id}_\mathcal B$, and assume that $\beta$
is still completely positive. Then the above equation becomes
$$
\omega_\beta^{[n]}(b)=b+(\beta\otimes\text{Id}_n)(\nu\otimes
\text{Id}_n)\left[(\mathcal X\otimes 1_n-\omega_\beta^{[n]}(b))^{-1}
\right],\quad b\in\mathbb H^+(M_n(\mathcal B)),n\in\mathbb N.
$$
This is precisely the subordination equation generalizing
the results of \cite[Lemma 4]{Biane2} to the operator-valued context:
if $\gamma_\beta$ is the centered operator-valued semicircular
distribution of variance $\beta$, then
\begin{equation}\label{semicirc}
G_{\nu\boxplus\gamma_\beta}^{[n]}=G_{\nu}^{[n]}\circ\omega_\beta^{[n]},
\quad n\in\mathbb N.
\end{equation}
There are deeper reasons for the similarity between the above formula 
and \eqref{eta}, reasons evidentiated in the case $\mathcal B=\mathbb C$
in \cite{BN1,BN2} and which are explored in \cite{ABFN} for arbitrary
$\mathcal B$. 

For the purposes of our present study, we specify the object of 
interest: the solution in $\mathbb H^+(\mathcal B)$ of the functional 
equation
\begin{equation}\label{prob}
\omega(b)=b+{\bf a}+\eta\left[(\mathcal X-\omega(b))^{-1}\right],
\quad b\in\mathbb H^+(\mathcal B),
\end{equation}
and its noncommutative extension to the noncommutative operator
upper half-plane, where $\mathcal B$ is an arbitrary unital 
C${}^*$-algebra, ${\bf a}={\bf a}^*\in\mathcal B$, $\mathcal B\langle
\mathcal X\rangle$ has a C${}^*$-algebra completion 
in which $\mathcal X=\mathcal X^*$, and $\eta\colon
\mathcal B\langle\mathcal X\rangle\to\mathcal B$ is bounded, completely
positive. The function $\omega$ is necessarily the right inverse of
\begin{equation}\label{H}
H(w)=w-{\bf a}-\eta\left[(\mathcal X-w)^{-1}\right],\quad
\in\mathbb H^+(\mathcal B).
\end{equation}
These facts were proved in \cite{ABFN} and from now on we will take them for granted.


\section{$(\Re\omega(\cdot+iq),\Im\omega(\cdot+iq))$ is the graph
of a function}

Let $\gamma_t$ be the semicircular (Wigner) law of variance 
$t\in(0,+\infty)$ and let $\mu$ be an arbitrary Borel probability 
measure on $\mathbb R$. In \cite[Lemma 2]{Biane2} it is shown that 
the imaginary part of the Cauchy-Stieltjes transform of 
$\mu\boxplus\gamma_t$ is, up to a factor of $-\pi^{-1}$, 
equal to the function $v_t(u)$ given as
$$
v_t(u):=\inf\left\{v\ge0|\int_\mathbb R\frac{td\mu(x)}{(u-x)^2+v^2}
\leq1\right\},
$$
and moreover, that this infimum is reached (i.e. $t\int_\mathbb R
\frac{d\mu(x)}{(u-x)^2+v_t(u)^2}=1$) whenever $v_t(u)>0$. Our next
proposition establishes a slightly weaker (and necessarily so) 
operator-valued counterpart of this result. We denote by 
$\mathcal B^{sa}$ the set of all selfadjoint elements of the 
C${}^*$-algebra $\mathcal B$, by $\mathcal B^+$ its subset of 
nonnegative elements, and by $\mathcal B^{++}$ the (open) subset of 
$\mathcal B^{sa}$ of strictly positive (i.e. nonnegative and 
invertible in $\mathcal B$) elements.

\begin{prop}\label{prop:3.1}
Let $\mathcal B$ be a C${}^*$-algebra, $\eta$ be a completely
positive map on the C${}^*$-completion of $\mathcal B\langle\mathcal X
\rangle$ and ${\bf a}$ be a selfadjoint element of $\mathcal B$.
For any fixed $q\in\mathcal B$, $q>0$, there exists a function
$v_q\colon\mathcal B^{sa}\to\mathcal B^{++}$ such that
$$
v_q(u)=q+\eta\left[((\mathcal X-u)v_q(u)^{-1}(\mathcal X-u)+v_q(u)
)^{-1}\right],
$$
for all $u\in\mathcal B^{sa}$. Moreover, the correspondence 
$u\mapsto v_q(u)$ is uniformly bounded (with a bound depending on 
$q,\eta$) and the restriction to $\mathcal B^{sa}$ of an analytic map.
\end{prop}

\begin{proof}
It is useful to clarify first the relation between our 
proposition and Equation \eqref{prob}: taking imaginary part in 
this equation and recalling that (i) $\mathcal B\langle\mathcal X
\rangle$ has a C${}^*$-algebra structure, (ii) $\mathcal X=
\mathcal X^*$, and (iii) $\eta$ is positive, provides us
with
$$
\Im\omega(b)=\Im b+\eta\left[\left((\mathcal X-\Re\omega(b))
(\Im\omega(b))^{-1}(\mathcal X-\Re\omega(b))+\Im\omega(b)
\right)^{-1}\right].
$$
We fix $\Im b=q>0$: then our proposition states that the imaginary part
of $\omega(b)$ is a continuous function of the real part of $\omega(b)
$. Here, of course, $\Re\omega(b)$ is viewed as an independent variable.

Thus, let us fix $q>0$. Define
$$
g_q\colon\mathcal B^{sa}\times\mathcal B^{++}\to\mathcal B^{++},
\quad
g_q(u,v)=q+\eta\left[((\mathcal X-u)v^{-1}(\mathcal X-u)+v
)^{-1}\right].
$$
For any $\epsilon=\epsilon^*\in\mathcal B$ and $v>0$, the relation
$(v+i\epsilon)^{-1}=(v+\epsilon v^{-1}\epsilon)^{-1}-i
(v+\epsilon v^{-1}\epsilon)^{-1}\epsilon v^{-1}$ implies that
\begin{eqnarray*}
\lefteqn{(\mathcal X-u)(v+i\epsilon)^{-1}(\mathcal X-u)+
v+i\epsilon=}\\
& & \mbox{}(\mathcal X-u)(v+\epsilon v^{-1}\epsilon)^{-1}
(\mathcal X-u)+v+i\left(\epsilon-(\mathcal X-u)(v+\epsilon v^{-1}\epsilon)^{-1}\epsilon v^{-1}(\mathcal X-u)\right)
\end{eqnarray*}
which guarantees that the real part (in the C${}^*$-algebra completion
of $\mathcal B\langle\mathcal X\rangle$) of 
$(\mathcal X-u)(v+i\epsilon)^{-1}(\mathcal X-u)+
v+i\epsilon$ is greater than $v$. This makes it invertible for 
{\em any} $\epsilon=\epsilon^*\in\mathcal B$, allowing
the extension of $g_q$ to $\mathcal B^{sa}\times(-i)\mathbb H^+(
\mathcal B)$, and, moreover, guarantees that
$\Re g_q(u,v+i\epsilon)\ge q$ for any $(u,v+i\epsilon)\in
\mathcal B^{sa}\times(-i)\mathbb H^+(\mathcal B)$. We have thus
re-written $g_q(u,\cdot)$ as a self-map of the noncommutative
operator right half-plane. Observe that $v>q/2$ implies $\|v^{-1}\|<
2\|q^{-1}\|$. Since for any selfadjoint $V$ the relation $\|(v+iV)^{-1}\|
\leq\|v^{-1}\|$ holds, the above relation implies
$$
\left\|\eta\left[\left((\mathcal X-u)(v+i\epsilon)^{-1}(\mathcal X-u)+
v+i\epsilon\right)^{-1}\right]\right\|\leq\|\eta\|\|v^{-1}\|<2\|\eta\|\|q^{-1}\|.
$$
Precisely the same argument as the one
from the proof of \cite[Theorem 8.4]{ABFN} shows that $g_q(u,\cdot)$ 
maps a bounded subdomain $\mathcal D$ of $\{w\in(-i)\mathbb H^+
(\mathcal B)\colon\Re w\ge q/2\}$, depending on $u$ and $q$, strictly 
inside itself. The Earle-Hamilton theorem \cite[Section 11.1]{Dineen} 
guarantees that $g_q(u,\cdot)$ has precisely one attracting fixed point 
in $(-i)\mathbb H^+(\mathcal B)+q$ for any $u\in\mathcal B^{sa}$, point 
which we call $v_q(u)$. Moreover, the function $w\mapsto 
g_q(u,w)$ is shown in the same reference to be a {\em strict} 
contraction in the Kobayashi metric, with the contraction 
coefficient depending continuously on the distance from 
$g_q(u,\mathcal D)$ to the complement of $\mathcal D$.
Thus, the dependence of the fixed point on $u,q$ 
is necessarily sequentially continuous (recall that the
dependence $u\mapsto g_q(u,v)$ is smooth - in fact analytic).
Since on any (norm)-bounded subset of $\mathbb H^+(\mathcal B)$
which is at a strictly positive (norm)-distance from $\mathcal B
\setminus\mathbb H^+(\mathcal B)$,
the topology generated by the Kobayashi metric coincides with the 
norm topology, this makes the correspondence $u\mapsto v_q(u)$
norm-continuous. As $g_q(u,v)>0$ for all $u=u^*,v>0$, the 
uniqueness of the attracting fixed point of $g_q(u,\cdot)$ 
implies that it necessarily belongs to $\mathcal B^{++}$.

In order to conclude, we must show that the correspondence 
$u\mapsto v(u)$ extends analytically to a neighbourhood of
$\mathcal B^{sa}$. Direct computations show that 
$g_{q\otimes1_2}^{[2]}$ can be extended to the set $\mathcal D_2$
of elements
$$
\left\{\left(\left(\begin{array}{cc}
u_1 & c\\
0 & u_2
\end{array}\right),\left(\begin{array}{cc}
w_1 & d\\
0 & w_2
\end{array}\right)\right)\colon u_1,u_2\in\mathcal B^{sa},
w_1,w_2\in(-i)\mathbb H^+(\mathcal B),c,d\in\mathcal B \right\}:
$$
the expressions of the $(1,1)$ and $(2,2)$ entries are
$g_q(u_1,w_1)$ and $g_q(u_2,w_2)$, respectively, the $(2,1)$
entry is zero, and the $(1,2)$ entry is
\begin{eqnarray}
\lefteqn{
\eta\left[((\mathcal X-u_1)w_1^{-1}(\mathcal X-u_1)+w_1)^{-1}\frac{}{}
\left[(\mathcal X-u_1)w_1^{-1}c+cw_2^{-1}(\mathcal X-u_2)\right.\right.
}\nonumber\\
& &\mbox{}-d+\left.\left.
(\mathcal X-u_1)w_1^{-1}dw_2^{-1}(\mathcal X-u_2)\right]
((\mathcal X-u_2)w_2^{-1}(\mathcal X-u_2)+w_2)^{-1}\frac{}{}\right]
\label{deriv}.
\end{eqnarray}
This makes $g_{q\otimes1_2}^{[2]}$ into a self-map of
$\mathcal D_2$. For $u_1,u_2,c$ fixed, 
$\left(\begin{array}{cc}
w_1 & d\\
0 & w_2
\end{array}\right)\mapsto g_{q\otimes1_2}^{[2]}\left(
\left(\begin{array}{cc}
u_1 & c\\
0 & u_2
\end{array}\right),\left(\begin{array}{cc}
w_1 & d\\
0 & w_2
\end{array}\right)\right)$ maps the set
$$
\mathcal D_1=\left\{\left(\begin{array}{cc}
w_1 & d\\
0 & w_2
\end{array}\right)\colon w_1,w_2\in(-i)\mathbb H^+(\mathcal B),
d\in\mathcal B\right\}
$$
into itself. We have noted that for fixed $u_1,u_2,c$, the relations
$\Re w_1,\Re w_2>q/2$ imply uniform  norm boundedness 
for the factors $((\mathcal X-u_j)w_j^{-1}(\mathcal X-u_j)+w_j)^{-1},
j\in\{1,2\}$ in the C${}^*$-algebra completion of $\mathcal B
\langle\mathcal X\rangle$, as well as of
$(\mathcal X-u_j)w_j^{-1}$ etc. However, this bound might be quite 
large, making the estimates on \eqref{deriv} uniform, but useless in 
terms of mapping a bounded subset of $\mathcal D_1$ into itself, thus
precluding another direct application of the Earle-Hamilton Theorem.
We shall go around this inconvenient fact.

It is clear that if 
$\left(\begin{array}{cc}
w_1 & d\\
0 & w_2
\end{array}\right)\mapsto g_{q\otimes1_2}^{[2]}\left(
\left(\begin{array}{cc}
u_1 & c\\
0 & u_2
\end{array}\right),\left(\begin{array}{cc}
w_1 & d\\
0 & w_2
\end{array}\right)\right)$
has a fixed point in $\mathcal D_1$, then the $(1,1)$
and $(2,2)$ entries of this fixed point must be 
$v_q(u_1)$ and $v_q(u_2)$, respectively. This puts a very
strong restriction on the $(1,2)$ entry of the fixed point:
it must be of the form 
\begin{eqnarray*}
\lefteqn{
\eta\left[((\mathcal X-u_1)v_q(u_1)^{-1}(\mathcal X-u_1)+v_q(u_1))^{-1}\frac{}{}\right.}\\
& &\mbox{}\times
\left[(\mathcal X-u_1)v_q(u_1)^{-1}c+cv_q(u_2)^{-1}(\mathcal X-u_2)
\right.\\
& &\mbox{}-d+\left.
(\mathcal X-u_1)v_q(u_1)^{-1}dv_q(u_2)^{-1}(\mathcal X-u_2)\right]\\
& &\mbox{}\times\left.
((\mathcal X-u_2)v_q(u_2)^{-1}(\mathcal X-u_2)+v_q(u_2))^{-1}\frac{}{}\right],
\end{eqnarray*}
for some $d\in\mathcal B$. This fixed point, if existing, must
depend linearly on $c$. Thus, we are allowed to re-scale $c$ as small
(in norm) as we desire. However, we are still inconvenienced
by the (implicit) requirement that the norm of the remaining
part of the expression above (the terms not containing $c$)
is {\em strictly} less than $\|d\|$. In order to address this 
problem, we need to enlarge the domain of definition of  
$g_{q\otimes1_2}^{[2]}$. For $u_1,u_2,c$ fixed as above, 
with the possible proviso that $c$ might be re-scaled 
(see equation \eqref{estimate} below), and $\varepsilon>0$, 
we consider the set
$$
\tilde{\mathcal D}_1^\varepsilon=
\left\{\left(\begin{array}{cc}
w_1 & d\\
m & w_2
\end{array}\right)\in M_2(\mathcal B)\colon
\Re\left(\begin{array}{cc}
w_1 & d\\
m & w_2
\end{array}\right)>\varepsilon1\otimes1_2
\right\}.
$$
The defining inequality of $\tilde{\mathcal D}_1^\varepsilon$ 
requires $\Re w_j>\varepsilon1$ and 
$\frac{d^*+m}{2}(\Re w_1-\varepsilon1)^{-1}\frac{d+m^*}{2}<
(\Re w_2-\varepsilon1)$. 
In order to study $g_{q\otimes1_2}^{[2]}$, we consider the expression
$$
\left[\left(\begin{array}{cc}
\mathcal X-u_1 & -c\\
0 & \mathcal X-u_2
\end{array}\right)
\left(\begin{array}{cc}
w_1 & d\\
m & w_2
\end{array}\right)^{-1}
\left(\begin{array}{cc}
\mathcal X-u_1 & -c\\
0 & \mathcal X-u_2
\end{array}\right)
+
\left(\begin{array}{cc}
w_1 & d\\
m & w_2
\end{array}\right)
\right]^{-1}
$$ 
which appears under $\eta\otimes\text{Id}_2$ in the formula 
of $g_{q\otimes1_2}^{[2]}$. Under the assumption that 
the argument belongs to $\tilde{\mathcal D}_1^\varepsilon
$, we determine under what conditions the element under the inverse 
has positive real part, and hence the whole expression above has 
positive real part (recall that $(-i)\mathbb H^+(\mathcal B)$ is 
invariant under taking inverse). We write
\begin{eqnarray*}
\lefteqn{\left(\begin{array}{cc}
\mathcal X-u_1 & -c\\
0 & \mathcal X-u_2
\end{array}\right)
\left(\begin{array}{cc}
w_1 & d\\
m & w_2
\end{array}\right)^{-1}
\left(\begin{array}{cc}
\mathcal X-u_1 & -c\\
0 & \mathcal X-u_2
\end{array}\right)}\\
&=&\left(\begin{array}{cc}
\mathcal X-u_1 & 0\\
0 & \mathcal X-u_2
\end{array}\right)
\left(\begin{array}{cc}
w_1 & d\\
m & w_2
\end{array}\right)^{-1}
\left(\begin{array}{cc}
\mathcal X-u_1 & 0\\
0 & \mathcal X-u_2
\end{array}\right)\\
& &\mbox{}-\left(\begin{array}{cc}
0 & c\\
0 & 0
\end{array}\right)
\left(\begin{array}{cc}
w_1 & d\\
m & w_2
\end{array}\right)^{-1}
\left(\begin{array}{cc}
\mathcal X-u_1 & 0\\
0 & \mathcal X-u_2
\end{array}\right)\\
& &\mbox{}-\left(\begin{array}{cc}
\mathcal X-u_1 & 0\\
0 & \mathcal X-u_2
\end{array}\right)
\left(\begin{array}{cc}
w_1 & d\\
m & w_2
\end{array}\right)^{-1}
\left(\begin{array}{cc}
0 & c\\
0 & 0
\end{array}\right)\\
& &\mbox{}+\left(\begin{array}{cc}
0 & c\\
0 & 0
\end{array}\right)
\left(\begin{array}{cc}
w_1 & d\\
m & w_2
\end{array}\right)^{-1}
\left(\begin{array}{cc}
0 & c\\
0 & 0
\end{array}\right),
\end{eqnarray*}
for $\left(\begin{array}{cc} w_1 & d\\ m & w_2 \end{array}\right)\in
\tilde{\mathcal D}_1^\varepsilon$.
It is clear that the first term on the right hand side above 
has real part greater than or equal to zero.
Since the real part of $\left(\begin{array}{cc}
w_1 & d\\
m & w_2
\end{array}\right)$
is greater than $\varepsilon$ times the unit of $M_2(\mathcal B)$, it 
follows that the norm of its inverse is no greater than $\varepsilon^{-1}
$. Thus, for all $c\in\mathcal B$ with 
\begin{equation}\label{estimate}
\|c\|<\min\left\{\frac12,\frac{\varepsilon^2}{4+16\|\mathcal X\|+8(\|u_1\|+\|u_2\|)}\right\},
\end{equation}
the norm of the sum of the real parts of the seconf
and third terms is strictly less than $\varepsilon/2$.
We conclude that 
$$
\Re\left[\left(\begin{array}{cc}
\mathcal X-u_1 & -c\\
0 & \mathcal X-u_2
\end{array}\right)
\left(\begin{array}{cc}
w_1 & d\\
m & w_2
\end{array}\right)^{-1}
\left(\begin{array}{cc}
\mathcal X-u_1 & -c\\
0 & \mathcal X-u_2
\end{array}\right)
\right]>-\frac\varepsilon21\otimes1_2.
$$
This guarantees that the real part of 
$$
\left(\begin{array}{cc}
\mathcal X-u_1 & -c\\
0 & \mathcal X-u_2
\end{array}\right)
\left(\begin{array}{cc}
w_1 & d\\
m & w_2
\end{array}\right)^{-1}
\left(\begin{array}{cc}
\mathcal X-u_1 & -c\\
0 & \mathcal X-u_2
\end{array}\right)+\left(\begin{array}{cc}
w_1 & d\\
m & w_2
\end{array}\right)
$$
is strictly greater than $\frac\varepsilon21\otimes1_2$. If we choose
$\varepsilon\in(0,1)$ sufficiently small so that 
$q>2\varepsilon1$ in $\mathcal B\langle\mathcal X\rangle$, then 
$\left(\begin{array}{cc}
w_1 & d\\
m & w_2
\end{array}\right)\mapsto g_{q\otimes1_2}^{[2]}\left(
\left(\begin{array}{cc}
u_1 & c\\
0 & u_2
\end{array}\right),\left(\begin{array}{cc}
w_1 & d\\
m & w_2
\end{array}\right)\right)$ maps $\tilde{\mathcal D}_1^\varepsilon$
in a bounded subset of itself which is at strictly positive
distance from the complement of $\tilde{\mathcal D}_1^\varepsilon$,
as shown above. 
The Earle-Hamilton Theorem \cite[Section 11.1]{Dineen} applies 
to provide a unique attracting fixed point in
$\tilde{\mathcal D}_1^\varepsilon$ for this correspondence. 
As noted above, upper diagonal matrices are mapped to upper
diagonal matrices. Thus, any iteration of 
$g_{q\otimes1_2}^{[2]}\left(
\left(\begin{array}{cc}
u_1 & c\\
0 & u_2
\end{array}\right),\quad\cdot\quad\right)$ that starts at an
upper diagonal matrix cannot converge to a matrix
that is not upper diagonal. Thus, we obtain a 
$d=d(u_1,u_2,c)$ such that 
$\left(\begin{array}{cc}
v_q(u_1) & d\\
0 & v_q(u_2)
\end{array}\right)$ is the attracting fixed point of the correspondence 
given just above. As argued above,
the dependence of the fixed point on the initial data $(u_1,u_2,c)$
is norm-continuous. 
With the notation $d(u_1,u_2,c)=
\Delta v_q(u_1,u_2)(c)$, justified by \cite[Section 2]{ncfound},
we obtain 
\begin{eqnarray}
\lefteqn{
\eta\left[((\mathcal X-u_1)v_q(u_1)^{-1}(\mathcal X-u_1)+v_q(u_1))^{-1}\frac{}{}\right.}\nonumber\\
& &\mbox{}\times
\left[(\mathcal X-u_1)v_q(u_1)^{-1}c+cv_q(u_2)^{-1}(\mathcal X-u_2)
\right.\nonumber\\
& &\mbox{}-\Delta v_q(u_1,u_2)(c)+\left.
(\mathcal X-u_1)v_q(u_1)^{-1}\Delta v_q(u_1,u_2)(c)
v_q(u_2)^{-1}(\mathcal X-u_2)\right]\nonumber\\
& &\mbox{}\times\left.
((\mathcal X-u_2)v_q(u_2)^{-1}(\mathcal X-u_2)+v_q(u_2))^{-1}\frac{}{}\right]=\Delta v_q(u_1,u_2)(c),\label{linIm}
\end{eqnarray}
for all $c\in\mathcal B$ of sufficiently small norm (estimated in
\eqref{estimate}), and, by linearity, for all $c\in\mathcal B$. 
Moreover, this same norm estimate \eqref{estimate} is seen to
be uniform for $u_1,u_2$ uniformly bounded. We conclude that
the correspondence $(u_1,u_2,c)\mapsto\Delta v_q(u_1,u_2)(c)$ is not 
only continuous, but also locally uniformly bounded when the norm 
topologies are considered on $\mathcal B^{sa}\times\mathcal B^{sa}
\times\mathcal B$ and $\mathcal B$. As shown in the same 
\cite[Section 2]{ncfound}, $\Delta v_q(u,u)(c)=\partial_uv_q(c)$.
We conclude that the correspondence $u\mapsto v_q(u)$ is in fact
$C^1$ in the Fr\'echet sense on $\mathcal B$.

We use next the property of $\Delta v_q(u_1,u_2)(c)$
to be an attracting fixed point for the map on the left 
hand side of \eqref{linIm}.
More precisely, we write the left hand side of \eqref{linIm}
as the sum of two linear maps (one of them, \eqref{Delta1}, is applied in \eqref{linIm} 
to $c$, the other, \eqref{Delta2}, to $\Delta v_q(u_1,u_2)(c)$):
\begin{eqnarray}
\lefteqn{\Delta_1g_{q}(u_1,u_2;v_q(u_1),v_q(u_2))(c)
=}\label{Delta1}\\
& &
\eta\left[((\mathcal X-u_1)v_q(u_1)^{-1}(\mathcal X-u_1)+v_q(u_1))^{-1}
\right.\nonumber\\
& &\mbox{}\times
\left[(\mathcal X-u_1)v_q(u_1)^{-1}c+cv_q(u_2)^{-1}(\mathcal X-u_2)
\right]\nonumber\\
& &\mbox{}\times\left.
((\mathcal X-u_2)v_q(u_2)^{-1}(\mathcal X-u_2)+v_q(u_2))^{-1}\right],
\nonumber
\end{eqnarray}
and
\begin{eqnarray}
\lefteqn{\Delta_2g_{q}(u_1,u_2;v_q(u_1),v_q(u_2))(d)
=}\label{Delta2}\\
& &
\eta\left[((\mathcal X-u_1)v_q(u_1)^{-1}(\mathcal X-u_1)+v_q(u_1))^{-1}
\right.\nonumber\\
& &\mbox{}\left[(\mathcal X-u_1)v_q(u_1)^{-1}d
v_q(u_2)^{-1}(\mathcal X-u_2)-d\right]\nonumber\\
& &\mbox{}\times\left.
((\mathcal X-u_2)v_q(u_2)^{-1}(\mathcal X-u_2)+v_q(u_2))^{-1}\right].
\nonumber
\end{eqnarray}
The correspondence that we iterate is 
\begin{equation}\label{mappe}
\left(\begin{array}{cc}
v_q(u_1) & d\\
0 & v_q(u_2)
\end{array}\right)\mapsto g_{q\otimes1_2}^{[2]}\left(
\left(\begin{array}{cc}
u_1 & c\\
0 & u_2
\end{array}\right),\left(\begin{array}{cc}
v_q(u_1) & d\\
0 & v_q(u_2)
\end{array}\right)\right).
\end{equation}
The right-hand side of this correspondence is
$$
\left(\begin{array}{cc}
v_q(u_1) & \Delta_1g_{q}
(u_1,u_2;v_q(u_1),v_q(u_2))(c)+
\Delta_2g_{q}(u_1,u_2;v_q(u_1),v_q(u_2))(d)\\
0 & v_q(u_2)
\end{array}\right)
$$
(recall that $g_q(u_j,v_q(u_j))=v_q(u_j)$ by the first part of the proof of Proposition
\ref{prop:3.1}). In order to save space, we denote just here
$S(\cdot)=\Delta_1g_{q}(u_1,u_2;v_q(u_1),v_q(u_2))(\cdot)$,
$T(\cdot)=\Delta_2g_{q}(u_1,u_2;v_q(u_1),v_q(u_2))(\cdot)$.
This expression makes clear that the $n^{\rm th}$
iteration of $g^{[2]}$ provides to its $(1,2)$ entry
$T^n(d)+\sum_{j=0}^{n-1}T^j(S(c))$. We conclude from the
existence of the limit as $n\to\infty$ of this expression
for any $c,d$ in a ball of small enough diameter (see Eq.
\eqref{estimate}) that $\left\|T^n(d)+\sum_{j=0}^{n-1}T^j(S(c))-\Delta 
v_q(u_1,u_2)(c)\right\|\to0$ as $n\to\infty$. In particular,
$\|T^n(d)\|\to0$ as $n\to\infty$.

Now we use again the essential property of the map \eqref{mappe}
to be a {\em strict} contraction in the Kobayashi metric,
with the contraction coefficient uniformly bounded away from one
when $u_1,u_2,c$ vary very little in norm. This makes the
above convergence to zero uniform in $c$ and $d$ for
$c,d$ in small enough norm-balls around zero in $\mathcal B$.
In particular, $\|T^n\|\to0$ as $n\to\infty$, imposing
$\sigma(T)\subset\mathbb D$. The claimed result follows by a
direct application of the analytic implicit function theorem.
\end{proof}

\begin{remark}\label{rm:3.2}
{\rm\ 
\begin{enumerate}
\item It should be noted that in the proof of the above proposition,
the noncommutative structure of the functions involved has not come
up in the proof of the existence and the norm-continuity of 
$u\mapsto v(u)$. In particular, for the proof of the existence,
boundedness and norm-continuity of $u\mapsto v(u)$, 
the requirement of complete positivity of the
linear map $\eta$ can be relaxed to simple positivity. Moreover, 
the proof of the analyticity of this correspondence involves only the
2-positivity of $\eta$. However, the
hypothesis of {\em complete} positivity is necessary, and sufficient, 
in order for the conclusion of Proposition \ref{prop:3.1} to hold
at all levels $n\in\mathbb N$, and thus, in the light of the motivation
for our investigation, it makes sense to keep it.

\item
The existence of $\omega(r+iq)$ for
any $r=r^*\in\mathcal B$, proved in \cite[Theorem 8.4]{ABFN},
as an attracting fixed point of $f_{r+iq}(w)=r+iq+{\bf a}
+\eta\left[(\mathcal X-w)^{-1}\right]$ guarantees that there are
pairs of points $(\Re\omega(r+iq),\Im\omega(r+iq))\in\mathcal B^{sa}
\times\mathcal B^{++}$ such that $g_q(\Re\omega(r+iq),\Im
\omega(r+iq))=\Im\omega(r+iq)$. The uniqueness of the fixed point
of $g_q(u,\cdot)$ guarantees that $v_q(\Re\omega(r+iq))=
\Im\omega(r+iq)$ whenever $u$ is of the form $\Re\omega(r+iq)$;
in particular, the set $\{(\Re\omega(r+iq),\Im\omega(r+iq))\colon
r\in\mathcal B^{sa}\}$ is the graph of a function defined on $\mathcal 
B^{sa}$ with values in $\mathcal B^{++}$.

\item
It is remarkable in this context that $g$ is the
first level of a noncommutative map having the
properties described in Proposition \ref{prop:3.1} at each level $n$
(in our proof, only levels $n=1$ and $n=2$ appear). 
Indeed, the noncommutative extension of $g$ is written as 
$g^{[n]}_{q\otimes1_n}(u,v)=q\otimes 1_n+(\eta\otimes{\rm Id}_n)
\left[((\mathcal X\otimes1_n-u)v^{-1}(\mathcal X\otimes1_n-u)
+v)^{-1}\right]$, for $u\in M_n(\mathcal B)^{sa}$, $v\in(-i)
\mathbb H^+(M_n(\mathcal B))$, $n\in\mathbb N$ 
($M_n(\mathcal B)^{sa},n\in\mathbb N$, is a noncommutative
set, but not an admissible one). This necessarily implies that
the fixed point is itself noncommutative: if $u_n=u\otimes1_n$,
then $v_{q\otimes1_n}^{[n]}(u\otimes1_n)=v_q(u)\otimes1_n$
(see \cite{AKV} for this fact, and for more properties of noncommutative 
fixed points).

\item When $\mathcal B$ is finite dimensional (a C${}^*$-algebra of
complex matrices) - the most important case in the study of distributions
of polynomials and rational functions in free random variables - it is much easier to show 
that $u\mapsto v_q(u)$ is analytic in the sense of Proposition \ref{prop:3.1}.
Indeed, the fact that $v_q(u)$ is an attracting fixed point
for $v\mapsto g_q(u,v)$ which is in the interior of the domain of 
$g_q(u,\cdot)$ implies that all eigenvalues of 
$\partial_vg_q(u,v_q(u))$ are of absolute value strictly less than
one. In order to see this, we write 
$$
g^{[2]}_{q\otimes1_2}\left(\left(\begin{array}{cc}
u & 0\\
0 & u
\end{array}\right),\left(\begin{array}{cc}
v & c\\
0 & v
\end{array}\right)\right)=\left(\begin{array}{cc}
g_q(u,v) & \partial_vg_q(u,v)(c)\\
0 & g_q(u,v)
\end{array}\right),
$$
and observe that for $c\in\mathcal B$ satisfying $(v^{-1/2}cv^{-1/2})
(v^{-1/2}cv^{-1/2})^*<4$ (hence for any $c\in\mathcal B$ of 
sufficiently small norm), the real part of $\left(\begin{array}{cc}
v & c\\
0 & v
\end{array}\right)$ is strictly positive in $M_2(\mathcal B)$.
Iterating the map $\left(\begin{array}{cc}
v_q(u) & c\\
0 & v_q(u)
\end{array}\right)\mapsto
g^{[2]}_{q\otimes1_2}\left(\left(\begin{array}{cc}
u & 0\\
0 & u
\end{array}\right),\left(\begin{array}{cc}
v_q(u) & c\\
0 & v_q(u)
\end{array}\right)\right)$ provides convergence in the
norm of $M_2(\mathcal B)$ to the fixed point
$\left(\begin{array}{cc}
v_q(u) & 0\\
0 & v_q(u)
\end{array}\right)$ as $n\to\infty$. Direct computation yields the 
formula
$
\left(\begin{array}{cc}
v_q(u) & \partial_vg_q(u,v_q(u))^{\circ n}(c)\\
0 & v_q(u)
\end{array}\right)
$
for the $n^{\rm th}$ iterate
of the map $\left(\begin{array}{cc}
v_q(u) & c\\
0 & v_q(u)
\end{array}\right)\mapsto
g^{[2]}_{q\otimes1_2}\left(\left(\begin{array}{cc}
u & 0\\
0 & u
\end{array}\right),\left(\begin{array}{cc}
v_q(u) & c\\
0 & v_q(u)
\end{array}\right)\right),$ so
we must have $\lim_{n\to\infty}\partial_vg_q(u,v_q(u))^{\circ n}(c)
=0$ for all $c\in\mathcal B$. This requires the spectrum
$\sigma(\partial_vg_q(u,v_q(u)))\subset\mathbb D$
(we have denoted by $\mathbb D$ the open unit disc in the
complex plane).
An application of the implicit function theorem for analytic 
functions provides the desired result, with a formula for the
derivative of $v_q$ given by
\begin{equation}\label{DerivForvq}
\partial_uv_q(u)=\left[\text{Id}_\mathcal B-
\partial_vg_q(u,v_q(u))\right]^{-1}\circ\partial_ug_q(u,v_q(u)).
\end{equation}
Note that this argument also required 2-positivity for $\eta$.
As a side benefit, note that $\sigma\left(\left[\text{Id}_\mathcal B
-\partial_vg_q(u,v_q(u))\right]^{-1}\right)\subseteq
\frac12-i\mathbb C^+.$

\end{enumerate}
}
\end{remark}

A further, rather straightforward, corollary of the above proposition
and remarks is recorded here.

\begin{cor}\label{cor:3.3}
For any $q>0$ in $\mathcal B$, the map $\mathcal B^{sa}\ni u\mapsto\Re
\omega(u+iq)\in\mathcal B^{sa}$ is bijective. For any $u=u^*$, the map $\mathcal B^{++}
\ni q\mapsto v_q(u)\in\mathcal B^{++}$ is injective.
\end{cor}

\begin{proof}
If $w=\Re\omega(u+iq)\in\Re\omega(\mathcal B^{sa}+iq)$, then the relation
\begin{eqnarray*}
\lefteqn{u=\Re\omega(u+iq)-{\bf a}-}\\
& & \eta\left[v_q(\Re\omega(u+iq))^{-1}(\mathcal X-\Re\omega(u+iq))\times\frac{}{}\right.\\
& & \left.\left((\mathcal X-\Re\omega(u+iq))v_q(\Re\omega(u+iq))^{-1}(\mathcal X-\Re\omega(u+iq))
+v_q(\Re\omega(u+iq))\right)^{-1}\right]
\end{eqnarray*}
indicates that $u\mapsto\Re\omega(u+iq)$ is the right inverse of the map
$$
w\mapsto \Phi_q(w)=w-{\bf a}-\eta\left[v_q(w)^{-1}(\mathcal X-w)\left((\mathcal X-w)
v_q(w)^{-1}(\mathcal X-w)+v_q(w)\right)^{-1}\right].
$$
This shows that $u\mapsto\Re\omega(u+iq)$ is injective.
It follows from Proposition \ref{prop:3.1} that $\Phi_q$ has an analytic extension to a small 
enough norm-neighbourhood (depending on $q>0$) of $\mathcal B^{sa}$. 
For elements $w$ whose inverses $w^{-1}$ are of small norm, it 
follows easily from the formula of $\Phi_q$ and the fixed-point equation satisfied by $v_q$ 
(see Proposition \ref{prop:3.1}) that $\|\Phi'_q(w)-\textrm{Id}\|$ is small. Thus, by 
Equation \eqref{prob}, the classical inverse function theorem for Banach spaces applied 
to $\Phi_q$ in a point of the form $w_M=\Re\omega(M1+iq)$ for $M\in(0,+\infty)$
sufficiently large provides a local inverse for $\Phi_q$ on a neighbourhood (in $\mathcal B$) of 
$w_M$ and guarantees that $u\mapsto\Re\omega(u+iq)$ maps a neighbourhood
of $1M$ onto a neighburhood of $w_M=\Re\omega(M1+iq)$. Now
\begin{eqnarray*}
u=\Phi_q(\Re\omega(u+iq))&\implies&\Re\omega(u+iq)=\Re\omega(\Phi_q(\Re\omega(u+iq))+iq)\\
&\implies& w=\Re\omega(\Phi_q(w)+iq),
\end{eqnarray*}
for all $w$ in an open set containing $w_M$. 
On the other hand, formula \eqref{H} indicates that, as functions of $u$,
both $\Re\omega(u+iq)$ and $\Im\omega(u+iq)$ have analytic extensions
to small enough (depending on $q>0$) norm-neighbourhoods in $\mathcal B$
of $\mathcal B^{sa}$. 
Thus, the fact that $\mathcal B^{sa}$ is a set of uniqueness for analytic maps 
in $\mathcal B$ implies that the relation $w=\Re\omega(\Phi_q(w)+iq)$ 
holds for all $w=w^*\in\mathcal B^{sa}$, so that $u\mapsto\omega(u+iq)$
is also surjective.

The second statement of the corollary is trivial.
\end{proof}

\section{The derivative of $\omega$}

\subsection{Spectrum of the derivative} For the case of $\mathcal B=\mathbb C$, it is shown in 
\cite[Theorem 4.6]{BB} that the difference quotient of $\omega$
satisfies the inequality
$$
\left|\frac{\omega(z_1)-\omega(z_2)}{z_1-z_2}\right|\geq\frac12,
\quad z_1,z_2\in\mathbb C^+\cup\mathbb R.
$$
This is shown by proving that $\Re\omega'(\alpha)>1/2$ for all
$\alpha\in\mathbb C^+$. The operator-valued counterpart of this
statement has the following form:
\begin{prop}\label{prop:4.1}
Let $\mathcal B$ be a unital C${}^*$-algebra and let $H$ and 
$\omega$ be defined as in \eqref{H} and \eqref{prob}.
For any $b_1,b_2\in\mathbb H^+(\mathcal B)$, the spectrum of
$\Delta\omega(b_1,b_2)$ as a linear operator from $\mathcal B$
to itself is included in $\{z\in\mathbb C\colon\Re z>1/2\}$. In particular,
$\Delta\omega(b_1,b_2)\colon\mathcal B\to\mathcal B$ is invertible for any $b_1,b_2\in
\mathbb H^+(\mathcal B)$.
\end{prop}

\begin{proof}
The proof is very similar in spirit to the proof of 
\cite[Theorem 4.6]{BB}. Consider $b_1,b_2\in\mathbb H^+(\mathcal B)$ 
and $c\in\mathcal B$ of sufficiently small norm so that
$\left(\begin{array}{cc}
b_1 & c\\
0 & b_2
\end{array}\right)\in\mathbb H^+(M_2(\mathcal B))$. We 
evaluate $\omega$ on this matrix in order to obtain
\begin{eqnarray*}
\left(\begin{array}{cc}
b_1 & c\\
0 & b_2
\end{array}\right)&=&H^{[2]}\left(\omega^{[2]}\left(\begin{array}{cc}
b_1 & c\\
0 & b_2
\end{array}\right)\right)\\
&=&\left(\begin{array}{cc}
H(\omega(b_1)) & \Delta H(\omega(b_1),\omega(b_2))\Delta\omega(b_1,b_2)(c)\\
0 & H(\omega(b_2))
\end{array}\right).
\end{eqnarray*}
This indicates that $\Delta H(\omega(b_1),\omega(b_2))\circ\Delta\omega
(b_1,b_2)=\text{Id}_\mathcal B$. As shown in the proof of 
\cite[Theorem 8.4]{ABFN}, the point 
$\omega^{[2]}\left(\begin{array}{cc}
b_1 & c\\
0 & b_2
\end{array}\right)\in\mathbb H^+(M_2(\mathcal B))$ is the unique
attracting fixed point of the self-map 
$$
f^{[2]}\colon w\mapsto \left(\begin{array}{cc}
b_1 & c\\
0 & b_2
\end{array}\right)+
\left(\begin{array}{cc}
{\bf a} & 0\\
0 & {\bf a}
\end{array}\right)
+(\eta\otimes\text{Id}_{M_2(\mathbb C)})\left[\left(
\left(\begin{array}{cc}
{\mathcal X} & 0\\
0 & {\mathcal X}
\end{array}\right)-w\right)^{-1}\right]
$$
of $\mathbb H^+(M_2(\mathcal B))$.
The methods used in the proof of Proposition \ref{prop:3.1} apply to
show that this map is a strict contraction in the 
Kobayashi metric. We conclude that the spectrum of
$\Delta f(\omega(b_1),\omega(b_2))$ is included in 
the open unit disc $\mathbb D$. However, 
$\Delta H(\omega(b_1),\omega(b_2))-\text{Id}_\mathcal B=
\Delta f(\omega(b_1),\omega(b_2)),$ which implies by the
definition of the spectrum that 
$\sigma(\Delta H(\omega(b_1),\omega(b_2)))\subset\mathbb D+1$.
Analytic functional calculus rules (see, for example, \cite[Section II.1.5]{Bruce}) provide
$$
\sigma(\Delta\omega(b_1,b_2))=
\sigma\left(\Delta H(\omega(b_1),\omega(b_2))^{-1}\right)\subset
\{z\in\mathbb C\colon\Re z>1/2\}.
$$
\end{proof}
As before, the proof of the above proposition only requires
$\eta$ to be $2$-positive.

We note a significant element: if we consider a $b_0$
in the {\em boundary} of $\mathbb H^+(\mathcal B)$
and we try to apply the implicit function theorem 
to the function $f(b,w)-w=b+{\bf a}+\eta\left[(\mathcal X
-w)^{-1}\right]-w$ around a point $(b_0,w_0)$, where
$w_0$ is a point in $\mathbb H^+(\mathcal B)\cap C(\omega,b_0)$
(where $C(\omega,b_0)$ denotes the set of limit 
points of $\omega$ at $b_0$), it turns out that
this is possible whenever $0\not\in
\sigma(\Delta H(w_0,w_0))$.
If $\mathcal B$ is finite dimensional, the set of
points $w_0$ with positive imaginary part that
satisfy such a condition is an analytic set. It turns
out that this analytic set has several properties of interest,
which become quite evident when one considers rather
the map $(w_1,w_2)\mapsto\Delta H(w_1,w_2)$, and which will
be investigated later.

\subsection{The  Julia-Carath\'eodory derivative}
We discuss  the Julia-Carath\'eodory derivatives at the 
distinguished boundary for the functions $\omega$ and $h$.
 We shall use  results and methods from \cite{B},
especially Propositions 3.1, 3.2 and Theorem 2.2. 
Assume that the there exists a sequence $\{z_n\}_{n\in\mathbb N}
\subset\mathbb C^+$ converging nontangentially to zero so that
the norm-limit $\lim_{n\to\infty}
\omega(\alpha+z_nv)=\omega(\alpha)$ exists
and is selfadjoint. This implies that 
$$
\lim_{n\to\infty}
\omega(\alpha+z_nv)=\alpha+{\bf a}+\lim_{n\to\infty}
\left(z_nv+\eta\left[(\mathcal X-\omega(\alpha+z_nv))^{-1}\right]\right),
$$
In particular, 
$$
\lim_{n\to\infty}
\eta\left[(\mathcal X-\omega(\alpha+z_nv))^{-1}\right]=
\omega(\alpha)-\alpha-{\bf a}
$$
in norm. If these statements are straightforward, providing more details 
on the properties of $\omega'(\alpha)$ requires a finer analysis.

Let us recall that for any $b\in\mathbb H^+(\mathcal B)$, $\omega(b)$ is the attracting fixed 
point of the map $f_b(w)=b+\eta\left[(\mathcal X-w)^{-1}\right]$, $w\in\mathbb H^+(\mathcal B)$.
Recall from \cite[Proposition 3.1]{B} that 
$$
\left\|(\Im f_b(w_1))^{-\frac12}(f_b(w_1)-f_b(w_2))(\Im f_b(w_2))^{-\frac12}\right\|
\leq\left\|(\Im w_1)^{-\frac12}(w_1-w_2)(\Im w_2)^{-\frac12}\right\|,
$$
for all $w_1,w_2,b\in\mathbb H^+(\mathcal B)$. In particular, if $w_1=\omega(b)=f_b(\omega(b))$,
then 
\begin{equation}\label{contraction}
\left\|(\Im\omega(b))^{-\frac12}(\omega(b)-f_b(w_2))(\Im f_b(w_2))^{-\frac12}\right\|
\leq\left\|(\Im\omega(b))^{-\frac12}(\omega(b)-w_2)(\Im w_2)^{-\frac12}\right\|,
\end{equation}
for all $w_2,b\in\mathbb H^+(\mathcal B)$. As in \cite[Proposition 3.2]{B},
for $r\in(0,+\infty)$ and $c\in\mathbb H^+(\mathcal B)$, define
$$
B^+_n(c,r)=\left\{a\in\mathbb H^+(M_n(\mathcal B))\colon\left\|(\Im a)^{-\frac12}(a-c\otimes 1_n)(\Im c\otimes 1_n)^{-\frac12}
\right\|\leq r\right\},
$$
and denote
$$
\mathring{B}^+_n(c,r)=
\left\{a\in\mathbb H^+(M_n(\mathcal B))\colon\left\|(\Im a)^{-\frac12}(a-c\otimes 1_n)(\Im c\otimes 1_n)^{-\frac12}\right\|<
r\right\}.
$$
It follows from \eqref{contraction} that $f_b^{[n]}(B^+_n(\omega(b),r))\subseteq
B^+_n(\omega(b),r)$ and $f_b^{[n]}(\mathring{B}^+_n(\omega(b),r))\subseteq
\mathring{B}^+_n(\omega(b),r)$ for any $r>0$. Observe that, while by its
definition $f_b^{[n]}(\mathring{B}^+_n(\omega(b),r))\subset\mathbb H^+
(M_n(\mathcal B))+b\otimes 1_n$, the inclusions above do not imply that
$\mathring{B}^+_n(\omega(b),r))-b\otimes 1_n\subset\mathbb H^+
(M_n(\mathcal B)).$
As noted also in \cite{B}, the defining inequality of $B^+_n(c,r)$ can be
re-writen as 
$$
(a-c\otimes 1_n)^*(\Im a)^{-1}(a-c\otimes 1_n)\leq r^2\Im c\otimes 1_n
$$
(and with $<$ for $\mathring{B}^+_n(c,r)$). Thus, under the additional assumption of the existence of 
a sequence $\{z_n\}_{n\in\mathbb N}\subset\mathbb C^+$ converging nontangentially to zero such that
$\ell:=\lim_{n\to\infty}\frac{\Im\omega(\alpha+z_nv)}{\|\Im\omega(\alpha+z_nv)\|}$
exists in the weak operator topology and is a strictly positive operator, we may take the
limit of $\mathring{B}^+_1(\omega(\alpha+z_nv),\|\Im\omega(\alpha+z_nv)\|^{-1/2})$
as $n\to\infty$ in the sense that 
\begin{eqnarray*}
H_1(\omega(\alpha),\ell)&\supseteq&
\bigcap_{n\in\mathbb N}\bigcup_{k\ge n}
{B}^+_1(\omega(\alpha+z_kv),\|\Im\omega(\alpha+z_kv)\|^{-1/2})\\
&\supseteq&\bigcap_{n\in\mathbb N}\bigcup_{k\ge n}
\mathring{B}^+_1(\omega(\alpha+z_kv),\|\Im\omega(\alpha+z_kv)\|^{-1/2})\\
&\supseteq&\bigcup_{n\in\mathbb N}\bigcap_{k\ge n}
\mathring{B}^+_1(\omega(\alpha+z_kv),\|\Im\omega(\alpha+z_kv)\|^{-1/2}),
\end{eqnarray*}
where 
\begin{eqnarray}
\lefteqn{ H_n(\omega^{[n]}(\alpha\otimes1_n),\ell\otimes 1_n)=}\nonumber\\
 & = & \left\{w\in\mathbb H^+(M_n(\mathcal B))\colon
(w-\omega^{[n]}(\alpha\otimes1_n))^*(\Im w)^{-1}(w-\omega^{[n]}(\alpha\otimes1_n))\leq\ell\otimes1_n\right\}\nonumber\\
\end{eqnarray}
plays the role of a noncommutative horodisc. (Note that under our current hypothesis 
$\ell$ might not belong to $\mathcal B$, but only to its enveloping W${}^*$-algebra.)
Indeed, the second and third inclusion are obvious.
We prove the first by contradiction. Assume that there is a point $x\in
\bigcap_{n\in\mathbb N}\bigcup_{k\ge n}
{B}^+_1(\omega(\alpha+z_kv),\|\Im\omega(\alpha+z_kv)\|^{-1/2})$ such that $x\not\in
H_1(\omega(\alpha),\ell)$, i.e. 
$(\Re x-\omega(\alpha))(\Im x)^{-1}(\Re x-\omega(\alpha))+\Im x\not\leq\ell$. This in
turn means that there is a norm-one vector $\xi$ in the Hilbert space on which 
the W${}^*$-envelope of $\mathcal B$ acts as a von Neumann algebra such that
$$
\left\langle[(\Re x-\omega(\alpha))(\Im x)^{-1}(\Re x-\omega(\alpha))+\Im x]\xi,\xi\right\rangle>
\langle\ell\xi,\xi\rangle.
$$
By definition, we know that for any $n\in\mathbb N$ there exists some $k\ge n$ such that
$x\in {B}^+_1(\omega(\alpha+z_kv),\|\Im\omega(\alpha+z_kv)\|^{-1/2})$, i.e. there exists
a subsequence of $\{z_n\}_n$ (call it for simplicity $\{z_{n_k}\}_k$) such that 
$$
(x-\omega(\alpha+z_{n_k}v))^*(\Im x)^{-1}
(x-\omega(\alpha+z_{n_k}v))\leq\frac{\Im\omega(\alpha+z_{n_k}v)\|}{\|\Im\omega(\alpha+z_{n_k}v)\|}.
$$
Apply this inequality to $\xi$ to obtain 
$$
\left\langle[(x-\omega(\alpha+z_{n_k}v))^*(\Im x)^{-1}
(x-\omega(\alpha+z_{n_k}v))]\xi,\xi\right\rangle\leq\left\langle
\frac{\Im\omega(\alpha+z_{n_k}v)\|}{\|\Im\omega(\alpha+z_{n_k}v)\|}\xi,\xi\right\rangle.
$$
The assumption that $\frac{\Im\omega(\alpha+z_nv)}{\|\Im\omega(\alpha+z_nv)\|}$ converges to $\ell$
in the weak operator topology as $n\to\infty$ allows us to conclude that the right-hand term of the above
converges to $\langle\ell\xi,\xi\rangle$ as $k\to\infty$.
Recall that, according to our hypothesis, $\lim_{k\to\infty}\|\Im\omega(\alpha+z_{n_k}v)\|=0$
and $\lim_{k\to\infty}\|\Re\omega(\alpha+z_{n_k}v)-\omega(\alpha)\|=0$. For the fixed $x$,
we then have
$$
\lim_{k\to\infty}\|(x-\omega(\alpha+z_{n_k}v))^*(\Im x)^{-1}
(x-\omega(\alpha+z_{n_k}v))
-(x^*-\omega(\alpha))(\Im x)^{-1}(x-\omega(\alpha))\|=0,
$$
so that the left-hand side of the above inequality converges to 
$$
\left\langle[(\Re x-\omega(\alpha))(\Im x)^{-1}(\Re x-\omega(\alpha))+\Im x]\xi,\xi\right\rangle,
$$
providing an immediate contradiction.
We now have all the necessary tools for proving the first main result of this subsection.
\begin{thm}\label{FullHorodiscs}
Assume that $\mathcal B$ is a von Neumann algebra, ${\bf a}={\bf a}^*\in\mathcal B$,
$\mathcal X=\mathcal X^*$ is a selfadjoint random variable in a von Neumann algebra which
contains $\mathcal B$ as a von Neumann subalgebra, and $\eta$ is a completely positive
map on the {\rm C}${}^*$-algebra completion of $\mathcal B\langle\mathcal X\rangle$. 
Assume also that $\omega$ is the noncommutative function provided by \eqref{prob},
$\alpha=\alpha^*\in\mathcal B$ and $\mathcal B\ni v>0$ are such that there exists a
sequence $\{y_n\}_{n\in\mathbb N}\subset\mathbb R^+$ converging to zero and 
$$
\omega(\alpha):=\lim_{n\to\infty}\omega(\alpha+iy_nv)
$$
exists in the norm topology and is a selfadjoint element of $\mathcal B$. 
If 
$$
\ell=\ell(v):=\lim_{n\to\infty}\frac{\Im\omega(\alpha+iy_nv)}{\|\Im\omega(\alpha+iy_nv)\|}
$$
exists in the norm topology and is a strictly positive operator, then
for any state $\varphi$ on $\mathcal B$ and any $u>0$,
$$
\liminf_{z\to0}\frac{\varphi\left(\Im\eta\left[(\mathcal X-\omega(\alpha)-zu)^{-1}\right]\right)}{\Im z}
<+\infty.
$$
\end{thm}

The conclusion of the above theorem does require that the convergence to $\ell$
takes place in the norm topology, a considerably stronger requirement than before,
but not surprising in the light of the results of, for example, \cite{Fan}.

\begin{proof} By passing to a subsequence if necessary, we assume without loss
of generality that $\{y_n\}_{n\in\mathbb N}$ is a strictly decreasing sequence in 
$(0,1)$. Since $f_b^{[n]}(B^+_n(\omega(b),r))\subseteq B^+_n(\omega(b),r)$ 
for any $b\in\mathbb H^+(M_n(\mathcal B))$ and $r>0$, it follows that
$\eta\left[(\mathcal X-\cdot)^{-1}\right]=f_b^{[n]}(\cdot)-b$ maps 
$B^+_n(\omega(b),r)$ into $B^+_n(\omega(b),r)-b$, the shift of 
$B^+_n(\omega(b),r)$ by $-b$. In particular, as $w\mapsto
\eta\left[(\mathcal X-w)^{-1}\right]$ sends the upper half-plane into its closure,
$B^+_n(\omega(\alpha+iy_nv),\|\Im \omega(\alpha+iy_nv)\|^{-1/2})$
is sent inside $(B^+_n(\omega(\alpha+iy_nv),\|\Im \omega(\alpha+iy_nv)\|^{-1/2})
-\alpha-iy_nv)\cap\overline{\mathbb H^+(M_n(\mathcal B))}$. 
According to \cite[Proposition 3.2]{B}, the sets 
$B^+_n(\omega(\alpha+iy_nv),\|\Im \omega(\alpha+iy_nv)\|^{-1/2})$ are bounded
by 
\begin{eqnarray*}
\lefteqn{\|\Re\omega(\alpha+iy_nv)\|+
\frac{1+\|\Im \omega(\alpha+iy_nv)\|+\sqrt{1+4\|\Im \omega(\alpha+iy_nv)\|}}{2}}\\
& & \mbox{}+\sqrt{\frac{1+\|\Im \omega(\alpha+iy_nv)\|+\sqrt{1+4\|\Im \omega(\alpha+iy_nv)\|}}{2}},
\end{eqnarray*}
quantities that tend to $\|\omega(\alpha)\|+2$ as $n\to\infty$. 
This makes the sets $H_n(\omega^{[n]}(\alpha\otimes1_n),\ell\otimes 1_n)$, $n\in\mathbb N$,
bounded in norm, uniformly in $n$ (see also \cite{AKV}).

It is for the next result that we need the strengthening of the 
convergence to $\ell$ to convergence in norm compared to 
the discussion before our present theorem. Under the hypothesis
of weak operator convergence we were able to show that $H_1(
\omega(\alpha),\ell)$ includes certain limsup and liminf of sets. 
Now we need to show in addition that 
$$
\bigcup_{n\in\mathbb N}\bigcap_{k\ge n}
\mathring{B}^+_1(\omega(\alpha+iy_kv),\|\Im\omega(\alpha+iy_kv)\|^{-1/2})
\supseteq\mathring{H}_1(\omega(\alpha),\ell),
$$
where $\mathring{H}_1(\omega(\alpha),\ell)=
\{w\in\mathbb H^+(\mathcal B)\colon(w-\omega(\alpha))^*(\Im w)^{-1}(w-\omega(\alpha))<
\ell\}.$
This inclusion is equivalent to showing that for any fixed $w\in\mathring{H}_1(\omega(\alpha),\ell)$,
there exists an $n_w\in\mathbb N$ such that 
$$
(w-\omega(\alpha+iy_{n}v))^*(\Im w)^{-1}
(w-\omega(\alpha+iy_{n}v))<\frac{\Im\omega(\alpha+iy_{n}v)}{\|\Im\omega(\alpha+iy_{n}v)\|}
$$
for all $n\ge n_w$. This follows quite easily, however. Indeed, for $w$ to satisfy the
{\em strict} inequality $(w-\omega(\alpha))^*(\Im w)^{-1}(w-\omega(\alpha))<
\ell$ it follows that we can find an $\varepsilon\in(0,1)$ such that 
$(w-\omega(\alpha))^*(\Im w)^{-1}(w-\omega(\alpha))+\varepsilon1<
\ell$. Then we only need to insure that our $n_w\in\mathbb N$ is sufficiently 
large in order to provide $\left\|\frac{\Im\omega(\alpha+iy_{n}v)}{\|\Im\omega(\alpha+iy_{n}v)\|}-\ell
\right\|<\frac\varepsilon8$ and $\|\omega(\alpha)-\Re\omega(\alpha+iy_{n}v)\|+
\|\Im\omega(\alpha+iy_{n}v)\|<\frac{\varepsilon}{8(1+\|w\|\|(\Im w)^{-1}\|)}$ for all
$n\ge n_w$. This proves the claimed inclusion.

Let {\color{red}{$h(w)=\eta\left[(\mathcal X-w)^{-1}\right]$}}. As seen above,
$h({B}^+_1(\omega(\alpha+iy_kv),\|\Im\omega(\alpha+iy_kv)\|^{-1/2})\subseteq
{B}^+_1(\omega(\alpha+iy_kv),\|\Im\omega(\alpha+iy_kv)\|^{-1/2})-\alpha-iy_kv$.
We have
\begin{eqnarray*}
\lefteqn{h(\mathring{H}_1(\omega(\alpha),\ell))}\\ 
& \subseteq & h\left(\bigcup_{n\in\mathbb N}\bigcap_{k\ge n}
\mathring{B}^+_1(\omega(\alpha+iy_kv),\|\Im\omega(\alpha+iy_kv)\|^{-1/2})\right)\\
& \subseteq & h\left(\bigcap_{n\in\mathbb N}\bigcup_{k\ge n}
{B}^+_1(\omega(\alpha+iy_kv),\|\Im\omega(\alpha+iy_kv)\|^{-1/2})\right)\\
& \subseteq & \bigcap_{n\in\mathbb N}h\left(\bigcup_{k\ge n}
{B}^+_1(\omega(\alpha+iy_kv),\|\Im\omega(\alpha+iy_kv)\|^{-1/2})\right)\\
& = & \bigcap_{n\in\mathbb N}\bigcup_{k\ge n}h\left(
{B}^+_1(\omega(\alpha+iy_kv),\|\Im\omega(\alpha+iy_kv)\|^{-1/2})\right)\\
& \subseteq & \bigcap_{n\in\mathbb N}\bigcup_{k\ge n}
\left[{B}^+_1(\omega(\alpha+iy_kv),\|\Im\omega(\alpha+iy_kv)\|^{-1/2})-\alpha-iy_kv\right]
\cap\overline{\mathbb H^+(\mathcal B)}\\
& = & \bigcap_{n\in\mathbb N}\bigcup_{k\ge n}
\left[{B}^+_1(\omega(\alpha+iy_kv)-\alpha,\|\Im\omega(\alpha+iy_kv)\|^{-1/2})-iy_kv\right]
\cap\overline{\mathbb H^+(\mathcal B)}\\
& = &\bigcap_{n>N}\bigcup_{k\ge n}\left[
{B}^+_1(\omega(\alpha+iy_kv)-\alpha,\|\Im\omega(\alpha+iy_kv)\|^{-1/2})-iy_kv\right]
\cap\overline{\mathbb H^+(\mathcal B)},
\end{eqnarray*}
for all $N\in\mathbb N$.
As $\{y_n\}_n$ is strictly decreasing, $-i[0,y_n]v\subsetneq-i[0,y_{n-1}]v$.
Observe the inclusion ${B}^+_1(\omega(\alpha-iy_kv)-\alpha,\|\Im\omega(\alpha+iy_kv)\|^{-1/2})-iy_kv
\subset{B}^+_1(\omega(\alpha+iy_kv)-\alpha,\|\Im\omega(\alpha+iy_kv)\|^{-1/2})-i
[0,y_{k-j}]v$ for all $j\in\{1,\dots,k\}$. Since 
$(\cup_\iota A_\iota)+K=\cup_\iota(A_\iota+K)$ for any subsets $(A_\iota)_\iota,K$ of
a topological vector space,
\begin{eqnarray*}
\lefteqn{\bigcup_{k\ge n}\left[B^+_1(\omega(\alpha+iy_kv)-\alpha,\|\Im\omega(\alpha+iy_kv)\|^{-1/2})
-i[0,y_{n-1}]v\right]}\\
& = & -i[0,y_{n-1}]v+\left[\bigcup_{k\ge n}
{B}^+_1(\omega(\alpha+iy_kv)-\alpha,\|\Im\omega(\alpha+iy_kv)\|^{-1/2})\right]\\
&\subseteq& -i[0,y_{N}]v+\left[\bigcup_{k\ge n}
{B}^+_1(\omega(\alpha+iy_kv)-\alpha,\|\Im\omega(\alpha+iy_kv)\|^{-1/2})\right],
\end{eqnarray*}
for all $n>N$.
Unfortunately in general $(A\cap V)+K\subsetneq(A+K)\cap(V+K)$. However,
if we have a decreasing sequence of norm-closed sets $A_1\supset A_2\supset A_3\supset\cdots$ and a 
norm-compact set $K$ in some Banach space $\mathcal Y$, then $\bigcap_{n\in\mathbb N}(A_n+K)=
(\bigcap_{n\in\mathbb N}A_n)+K$. Indeed, let $x\in\bigcap_{n\in\mathbb N}(A_n+K)$. This means that
$x\in A_n+K$ for all $n\in\mathbb N$, so that there are $x_n\in A_n$ and $\kappa_n\in K$ such
that $x=x_n+\kappa_n$. Since $K$ is norm-compact, $\{\kappa_n\}_{n\in\mathbb N}$ has a 
norm-convergent subsequence $\{\kappa_{n_j}\}_{j\in\mathbb N}$, converging to $\kappa\in K$.
Thus, $\lim_{j\to\infty}x_{n_j}=x-\lim_{j\to\infty}\kappa_{n_j}=x-\kappa\in A_n$ for all 
$n\in\mathbb N$. So $x=(x-\kappa)+\kappa\in(\bigcap_{n\in\mathbb N}A_n)+K$, as claimed.
We apply this to $K=-i[0,y_{N}]v$ and $A_n=\cup_{k\ge n}{B}^+_1(\omega(\alpha+iy_kv)-\alpha,\|
\Im\omega(\alpha+iy_kv)\|^{-1/2})$ to conclude that
\begin{eqnarray*}
\lefteqn{h(\mathring{H}_1(\omega(\alpha),\ell))}\\
& \subseteq &\bigcap_{n>N}\left\{-i[0,y_{N}]v+\left[\bigcup_{k\ge n}
{B}^+_1(\omega(\alpha+iy_kv)-\alpha,\|\Im\omega(\alpha+iy_kv)\|^{-1/2})\right]
\right\}\\
& & \mbox{}\cap\overline{\mathbb H^+(\mathcal B)}\\
& = & \left[-i[0,y_{N}]v+\bigcap_{n>N}\bigcup_{k\ge n}
{B}^+_1(\omega(\alpha+iy_kv)-\alpha,\|\Im\omega(\alpha+iy_kv)\|^{-1/2})\right]\cap\overline{\mathbb 
H^+(\mathcal B)}\\
& \subseteq & \left[-i[0,y_{N}]v+H_1(\omega(\alpha)-\alpha,\ell)\right]\cap\overline{\mathbb 
H^+(\mathcal B)},
\end{eqnarray*}
for any $N\in\mathbb N$. As $N$ is arbitrary and $y_N\to0$ when $N\to\infty$, we obtain
$$
h(\mathring{H}_1(\omega(\alpha),\ell) \subseteq\overline{H_1(\omega(\alpha)-\alpha,\ell)}.
$$
Replacing the denominator in the definition $\ell=\lim_{n\to\infty}
\frac{\Im\omega(\alpha+iy_nv)}{\|\Im\omega(\alpha+iy_nv)\|}$ by $\kappa\|\Im\omega(\alpha+iy_nv)\|
$ for any $\kappa\in(0,+\infty)$ (which corresponds to changing radii in the definition of $B^+$ to 
$(\kappa\|\Im\omega(\alpha+iy_nv)\|)^{-1/2}$) yields
\begin{equation}\label{Horo-h}
h(\mathring{H}_1(\omega(\alpha),\ell/\kappa))\subseteq
\overline{H_1(\omega(\alpha)-\alpha,\ell/\kappa)}.
\end{equation}
Pick an arbitrary $u>0$. There exists $\epsilon>0$ such that $\omega(\alpha)+i\epsilon u\in
\mathring{H}_1(\omega(\alpha),\ell)$. Indeed, this relation is equivalent to
$$
(\omega(\alpha)+i\epsilon u-\omega(\alpha))^*(\epsilon u)^{-1}
(\omega(\alpha)+i\epsilon u-\omega(\alpha))<\ell,
$$
i.e. $\epsilon u<\ell.$ Since $\ell>0$, the existence of such an $\epsilon$ is guaranteed (for ex. any 
$0<\epsilon<(\|u\|\|\ell^{-1}\|)^{-1}$ will do). Fix such an $\epsilon$ and call it $\epsilon_0$.
It follows from the definition of $\mathring{H}_1$ that $(\epsilon_0/\kappa)u\in\mathring{H}_1
(\omega(\alpha),\ell/\kappa)$. By applying \cite[Relation (14)]{B} with $r=(\kappa\|
\Im\omega(\alpha+iy_nv)\|)^{-1/2}$ and letting $n\to\infty$ we obtain that $\|\Im w\|<\kappa^{-1}$ for 
any $w\in\mathring{H}_1(\omega(\alpha),\ell/\kappa)$ and $\|\Im w\|\le\kappa^{-1}$ for any $w\in
{H}_1(\omega(\alpha),\ell/\kappa)$. Now our theorem follows: $h(\omega(\alpha)+i(\epsilon_0/\kappa)u)
\in\overline{H_1(\omega(\alpha),\ell/\kappa)}$ implies $\|\Im h(\omega(\alpha)+i(\epsilon_0/\kappa)u)
\|\leq\kappa^{-1}$ for all $\kappa>0$, so that
$$
\frac{\varphi(\Im h(\omega(\alpha)+i(\epsilon_0/\kappa)u))}{\kappa^{-1}}\leq1.
$$
Taking $\epsilon_0\kappa^{-1}=y\downarrow0$ provides us with
$$
\frac{\varphi(\Im h(\omega(\alpha)+iyu))}{y}\leq\epsilon_0^{-1},
$$
concluding the proof. We observe that, not surprisingly, the optimal bound $\epsilon_0^{-1}$ is given 
by the largest $\epsilon_0$ for which $\omega(\alpha)+i\epsilon_0 u\in\mathring{H}_1(\omega(\alpha),
\ell)$.
\end{proof}
Recall that $f_\alpha(w)=\alpha+{\bf a}+\eta\left[(\mathcal X-w)^{-1}\right]=\alpha+{\bf a}+h(w)$.
\begin{cor}
Assume that $\omega$ satisfies the hypotheses of Theorem \ref{FullHorodiscs} at $\alpha=\alpha^*$. 
Then for any $v>0$, we have that
$$
\lim_{y\downarrow0}f'_\alpha\left(\omega(\alpha)+iyv
\right)
$$
exists and is bounded, and
$$
\left\|\lim_{y\downarrow0}
f'_\alpha\left(\omega(\alpha)+iy\ell\right)(\ell)
\right\|\leq1.
$$
\end{cor}
\begin{proof}
We use the notations and definitions from Theorem \ref{FullHorodiscs} and its proof. 
The conclusion of Theorem \ref{FullHorodiscs} implies, according to \cite[Theorem 2.2 (1)]{B}, that
the limit $\lim_{y\downarrow0}h'(\omega(\alpha)+iyv)$ exists and is bounded for all $v>0$. Then
$$
\lim_{y\downarrow0}\frac{\Im h(\omega(\alpha)+iy\ell)}{y}=h'(\omega(\alpha))_\ell\left(\ell\right),
$$
where $h'(\omega(\alpha))_\ell(b)=\lim_{y\downarrow0}h'(\omega(\alpha)+iy\ell)(b),$ $b\in\mathcal B$.
As it follows from the proof of Theorem \ref{FullHorodiscs}, since the optimal $\epsilon_0$
for $u=\ell$ is one, we have 
$$
\left\|\lim_{y\downarrow0}h'\left(\omega(\alpha)+iy\ell\right)(\ell)\right\|\leq1.
$$ 
Since $f'_\alpha=h'$, this proves the second statement.
\end{proof}

\subsection*{Acknowledgement} I am very grateful to Hari Bercovici for having
provided valuable feedback on an earlier version of this paper.

\end{document}